\documentclass[12pt]{amsart}
\usepackage[english]{babel}
\usepackage[utf8x]{inputenc}
\usepackage{amsfonts}
\usepackage{amsmath}
\usepackage{longtable}
\usepackage{graphicx}
\usepackage{enumerate}
\usepackage{tikz-cd}
\usepackage{bbm}
\usepackage{mathtools}
\usepackage{amsthm}
\usepackage{amssymb}
\usepackage{url,smartref}
\usepackage{mathrsfs} 
\usepackage{tikz,comment}
\usepackage{extpfeil,bm,fancyhdr,mathbbol}
\usepackage[top=1in, bottom=1.25in, left=1.25in, right=1.25in]{geometry}
\usetikzlibrary{shapes, arrows}
\tikzstyle{terminator} = [rectangle, draw, text centered, rounded corners, minimum height=2em]

\usepackage{nameref}
\usepackage{enumitem, hyperref}

\makeatletter
\def\namedlabel#1#2{\begingroup
   \def\@currentlabel{#2}%
   \label{#1}\endgroup
}
\makeatother

\makeatletter
\newcommand{\labitem}[2]{%
\def\@itemlabel{\textbf{#1}}
\item
\def\@currentlabel{#1}\label{#2}}
\makeatother

\numberwithin{equation}{section}

\usepackage{multirow}
\usepackage{subfigure}
\usepackage{algorithm}
\usepackage{algorithmic}
\usepackage{comment,mathrsfs}

\title{The growth of Tate--Shafarevich groups of $p$-supersingular elliptic curves over anticyclotomic $\bbZ_p$-extensions at inert primes}
\author{Erman I{\c S}IK and Antonio Lei}


\address{University of Ottawa, Department of Mathematics and Statistics, STEM Complex, 150 Louis-Pasteur Pvt, Ottawa, Ontario, K1N 6N5, Canada}
\email{\url{eisik@uottawa.ca},  \url{antonio.lei@uottawa.ca}}
\urladdr{\url{https://sites.google.com/view/erman-isik/}\\\url{https://antoniolei.com/}}

\newcommand{\orange}[1]{{\color{orange} \sf  #1}}

\newcommand{\mylabel}[2]{#2\def\@currentlabel{#2}\label{#1}}
\DeclareFontFamily{U}{wncy}{}
    \DeclareFontShape{U}{wncy}{m}{n}{<->wncyr10}{}
    \DeclareSymbolFont{mcy}{U}{wncy}{m}{n}
    \DeclareMathSymbol{\Sh}{\mathord}{mcy}{"58}

\newcommand{\Zp}{\bbZ_p}

\newcommand{\Qp}{\bbQ_p}

\newcommand{\image}{\mathrm{Image}}
\theoremstyle{definition}
\newtheorem{definition}{Definition}[section]
\newtheorem{lemma}[definition]{Lemma}
\newtheorem{theorem}[definition]{Theorem}
\newtheorem{prop}[definition]{Proposition}
\newtheorem{corollary}[definition]{Corollary}
\newtheorem{conj}[definition]{Hypothesis}

\newtheorem{remark}[definition]{Remark}

\newtheorem{thmx}{Theorem}

\DeclareFontFamily{U}{wncy}{}
    \DeclareFontShape{U}{wncy}{m}{n}{<->wncyr10}{}
    \DeclareSymbolFont{mcy}{U}{wncy}{m}{n}
    \DeclareMathSymbol{\Sh}{\mathord}{mcy}{"58}

\newcommand{\hatE}{\widehat{E}}

\newcommand{\ccS}{\mathcal{S}}
\newcommand{\ccF}{\mathcal{F}}

\newcommand{\ccO}{\mathcal{O}}

\newcommand{\ccX}{\mathcal{X}}
\newcommand{\ccY}{\mathcal{Y}}

\newcommand{\bbF}{{\mathbb F}}

\newcommand{\bbQ}{{\mathbb Q}}

\newcommand{\bbZ}{{\mathbb Z}}

\newcommand{\bff}{{\mathbf f}}

\newcommand{\bfz}{{\mathbf z}}

\newcommand{\fram}{{\mathfrak m}}

\newcommand{\fraq}{{\mathfrak q}}

\newcommand{\frakS}{{\mathfrak S}}

\newcommand{\mat}{ M_{2\times 2}(\Lambda)}

\newcommand{\Ltwo}{\Lambda^{\oplus 2}}
\newcommand{\An}{A_{(n)}}

\newcommand{\gal}{{\rm Gal}}
\newcommand{\tr}{{\rm Tr}}

\newcommand{\Hom}{{\rm Hom}}

\newcommand{\Sel}{{\rm Sel}}

\newcommand{\cyc}{{\rm cyc}}

\newcommand{\im}{{\rm Im}}
\newcommand{\loc}{{\rm loc}}
\newcommand{\col}{{\rm Col}}

\newcommand{\Char}{{\rm char}}

\newcommand{\iw}{{\rm Iw}}

\newcommand{\ord}{{\rm ord}}

\newcommand{\rank}{{\rm rank}}

\newcommand{\SL}{{\rm SL}}

\newcommand{\coker}{{\rm coker}}

\newcommand{\length}{{\rm length}}
\newcommand{\cor}{{\rm Cor}}
\newcommand{\res}{{\rm Res}}
\newcommand{\ext}{{\rm Ext}}
\newcommand{\corank}{{\rm corank}}

\newcommand{\pr}{{\rm pr}}
\newcommand{\tomg}{{\widetilde{\omega}}}

\newcommand{\bu}{\mathbf{u}}

\usepackage[OT2,T1]{fontenc}
\DeclareSymbolFont{cyrletters}{OT2}{wncyr}{m}{n}
\DeclareMathSymbol{\Sha}{\mathalpha}{cyrletters}{"58}

\definecolor{Green}{rgb}{0.0, 0.5, 0.0}

\newcommand{\ZZ}{\mathbb{Z}}
\begin{document}
\maketitle

\begin{abstract}
    Let $E$ be an elliptic curve defined over $\bbQ$, and let $K$ be an imaginary quadratic field. Consider an odd prime $p$ at which $E$ has good supersingular reduction with $a_p(E)=0$ and which is inert in $K$. Under the assumption that the signed Selmer groups are cotorsion modules over the corresponding Iwasawa algebra, we prove that the Mordell--Weil ranks of $E$ are bounded over any subextensions of the anticyclotomic $\bbZ_p$-extension of $K$. Additionally, we provide an asymptotic formula for the growth of the $p$-parts of the Tate--Shafarevich groups of $E$ over these extensions.
\end{abstract}

\section{Introduction}

Let $E$ be an elliptic curve defined over $\bbQ$ and $p$ be an odd prime of good reduction.  Let $\bbQ_\cyc$ denote the cyclotomic $\bbZ_p$-extension of $\bbQ$, and $\bbQ_{(n)}$  be the unique subextension of $\bbQ_\cyc$ of degree $p^n$ for $n\geq 0$. Suppose that the $p$-primary component of the Tate--Shafarevich group $ \Sha\left(E/\bbQ_{(n)}\right)$ is finite for all $n$, and let $p^{e_n} =\big| \Sha(E/K_n )[p^\infty]\big|$. The variation of $e_n$ in $n$ depends on whether the elliptic curve $E$ has ordinary or supersingular reduction at $p$.

\vspace{2mm}

When $E$ has ordinary reduction at $p$, it follows from \cite[Theorem 17.4]{kato04} that the $p$-primary Selmer group of $E$ over $\bbQ_\cyc$ is cotorsion over the associated Iwasawa algebra. It then follows from Mazur's control theorem \cite{Maz72} that for $n\gg0$, 
 \begin{equation*}
    e_n-e_{n-1}  = \lambda + (p^n-p^{n-1}) \mu - r_\infty,
 \end{equation*}
 where $r_\infty$ is the rank of $E$ over $\bbQ_\cyc$ (which is finite by \cite{kato04,roh84}), and $\lambda$ and $\mu$ are the Iwasawa invariants of the Pontryagin dual of the $p$-primary Selmer group of $E$ over $\bbQ_\cyc$. 

 \vspace{2mm}

 When $E$ has supersingular reduction at $p$, the case where $a_p(E)=0$ has been studied by Kurihara \cite{kur02}, Kobayashi \cite{kob03} and Pollack \cite{Pol05}.  For $n\gg 0$, we have the following formula:
 \begin{equation*}
     e_n - e_{n-1}  = \deg\tomg_n^\mp + \lambda_{\pm} + \mu_{\pm}(p^n-p^{n-1}) - r_\infty,
 \end{equation*}
 where $\tomg_n^\mp$ is the product of certain cyclotomic polynomials, $r_\infty$ is the rank of $E$ over $\bbQ_\cyc$, $\lambda_\pm$ and $\mu_\pm$ are the Iwasawa invariants of the cotorsion signed Selmer groups of $E$ over $\bbQ_\cyc$, and the sign $\pm$ depends on the parity of $n$.

\vspace{2mm}

In \cite{Spr13,Spr17}, Sprung obtained an analogous formula for the $p$-supersingular elliptic curves with $a_p(E)\neq 0$ and abelian varieties of GL(2)-type. See also \cite{LLZ17} for related work on upper bounds on the growth of the Bloch--Kato--Shafarevich--Tate groups for higher-weight modular forms. Furthermore, the growth of Mordell--Weil ranks and the Tate--Shafarevich groups of an elliptic curve over the cyclotomic $\bbZ_p$-extension of certain number fields have been studied in \cite{LL22}.

\vspace{2mm}
The focus of this article is analogous results for the anticyclotomic $\Zp$-extension of an imaginary quadratic field $K$. Suppose that $p$ splits in $K$ and the two primes above $p$ splits completely in the anticyclotomic $\Zp$-extension of $K$.  
Under certain hypotheses on the vanishing of the Mordell--Weil rank and the $p$-part of the Tate--Shafarevich group of $E$ over $K$, Iovita and Pollack \cite[Theorem 5.1]{ip06} derived a formula of the growth of the Tate--Shafarevich groups over anticyclotomic tower that mimics the cyclotomic case. Their hypotheses imply that both the plus and minus Selmer groups have trivial Iwasawa invariants, meaning that the growth of the Tate--Shafarevich groups is minimal. Their method can potentially be extended to the case where $p$ is inert in $K$ under similar stringent hypotheses. We avoid such hypotheses in this article. Instead, we develop a method that allows the plus and minus Selmer groups to have non-trivial Iwasawa invariants.

\vspace{2mm}

 If we assume in addition that $p \geq 5$ and that $(E,K)$ satisfies the Heegner hypothesis, {\c C}iperiani proved in \cite{Cip09} that the Tate–Shafarevich group of $E$ has trivial corank over the Iwasawa algebra associated with the anticyclotomic $\bbZ_p$-extension of $K$. In \cite{matar21,llm23}, a more precise formula of the growth of the Tate--Shafarevich groups inside this tower is given. More recently, Burungale--Kobayashi--Ota and M\"uller \cite{bko24,muller25} studied this question when $E$ has complex multiplication by an order in $K$. The methods in these works are specifically developed for the setting where at least one of the plus and minus Selmer groups is not cotorsion. The main goal of the present article is to study the growth of the Tate–Shafarevich groups when both Selmer groups are cotorsion, for which the methods in the aforementioned works do not extend directly.

\subsection{Statement of the main theorem} 
Let $K$ be an imaginary quadratic field such that $p$ is \textit{inert} in $K$. Assume that $p$ does not divide the class number $h_K$ of $K$ and $a_p(E)=0$. Let $K_\infty$ denote the anticyclotomic $\bbZ_p$-extension of $K$. Note that the unique prime of $K$ above $p$ is totally ramified in $K_\infty$. For an integer $n \geq 0$, we write $K_n$ for the unique subextension of $K_\infty$ such that $[K_n : K] = p^n$.

\vspace{2mm}

 The Galois group of $K_\infty$ over $K$ is denoted by $\Gamma$. In addition, write $\Gamma_n = \gal(K_\infty/K_n)$ and $G_n = \gal(K_n/K)$. We fix once and for all a topological generator $\gamma$ of $\Gamma$. Let $\Lambda$ denote the Iwasawa algebra $\bbZ_p[[\Gamma]]=\displaystyle\varprojlim_n\Zp[G_n]$, which we shall identify with the power series ring $\bbZ_p[[X]]$ by sending $\gamma-1$ to $X$. Let $(-)^\vee := \Hom(-,\bbQ_p/\bbZ_p)$ denote the Pontryagin duality functor.

\vspace{2mm}

 In this article, we prove that the Mordell--Weil rank of $E$ over $K_\infty$ is bounded assuming that the signed Selmer groups $\Sel^\pm(E/K_\infty)$ are $\Lambda$-cotorsion. Furthermore, we derive an asymptotic formula for the growth of the $p$-parts of the Tate--Shafarevich groups $ \Sha(E/K_n )$ of $E$ in terms of the Iwasawa invariants of $\Sel^\pm(E/K_\infty)^\vee$. These Selmer groups have recently been studied in \cite{bbl24}, and are similar to those defined by Rubin in \cite{rub87loc} for CM elliptic curves (see Definition~\ref{def:pmSel} for the precise definition and Remark~\ref{rk:pmGroups} for a discussion on how the plus and minus subgroups used to define these groups compare with the counterparts used to define  Kobayashi's plus and minus Selmer groups in \cite{kob03}). The main result of this article is the following theorem.

 \vspace{2mm}

\begin{thmx}\label{maintheorem}
Let $E$ be an elliptic curve defined over $\bbQ$ and $p$ an odd prime where $E$ has good supersingular reduction with $a_p(E)=0$.
Let $K$ be an imaginary quadratic field such that $p$ is inert in $K$. Assume that $p$ does not divide the class number $h_K$. Let $K_\infty$ denote the anticyclotomic $\bbZ_p$-extension of $K$.
  Assume that $\Sel^\pm(E/K_\infty)^\vee$ are both $\Lambda$-torsion.  Then, we have:
  \begin{enumerate}[label=(\roman*)]
        \item $\rank_\bbZ E(K_n)$ is bounded independently of $n$;
        
\vspace{1.5mm}
        
        \item Assume that $\Sha(E/K_n)[p^
        \infty]$ is finite for all $n\geq 0$.  Let $p^{e_n} =\big| \Sha(E/K_n )[p^\infty]\big|$. Define $\displaystyle r_\infty=\sup_{n\ge0}\left\{ \rank_\bbZ E(K_n)\right\}$. Let $\lambda_\pm$ and $\mu_\pm$ be the Iwasawa invariants of $\Sel^\pm(E/K_\infty)^\vee$. Then, for $n\gg0$, 
      \[e_n-e_{n-1}=\begin{cases}
      2s_{n-1}  + \lambda_- + \phi(p^n)\mu_-   -r_\infty    & \;\;\; \text{if } n  \text{ is odd,}  \\
          2s_{n-1}+\lambda_+ + \phi(p^n)\mu_+ -r_\infty  & \;\;\; \text{if } n \text{ is even,}
        \end{cases} \]
        where $\phi(p^n)=p^n-p^{n-1}$, and $s_n=\sum_{k=1}^{n}(-1)^{n-k}p^k$ for $n\geq 0$.
    \end{enumerate}
\end{thmx}

 \vspace{2mm}

\begin{remark}
 Under certain hypotheses, it is known that $\Sel^\pm(E/K_\infty)^\vee$ are $\Lambda$-torsion; see \cite{BLV,bbl24} and Remark~\ref{cotorsionresult} for a more detailed discussion. Our result is complementary to the recent works of \cite{bko24,muller25}, where $E$ is assumed to have complex multiplication by an order in $K$. In their setting, there exists exactly one element $\bullet\in\{+,-\}$ (the choice of which depends on the root number of $E$) such that $\Sel^\bullet(E/K_\infty)^\vee$ is $\Lambda$-torsion. Furthermore, $\rank_\bbZ E(K_n)$ is unbounded as $n\rightarrow\infty$.
\end{remark}

\subsection{Organization}

In Section \ref{localpoints}, we review several plus and minus objects. We first review a system of local points introduced in \cite[\S 2]{bko23}, which is used to construct signed (or plus and minus) Coleman maps following \cite{kob03}. These maps are then utilized to define the signed Selmer groups that appear in the statement of Theorem~\ref{maintheorem}. After analyzing the images of these Coleman maps, we review results on certain global cohomology groups under the assumption that the signed Selmer groups are cotorsion. We review the definition of Kobayashi ranks for projective systems of $\mathbb{Z}_p$-modules in Section \ref{kobayashiranks}. We derive a formula for the Kobayashi ranks of modules arising from specific $2 \times 2$ matrices defined over $\Lambda$, which is central to our proof of Theorem~\ref{maintheorem} and is the principal novelty of this article. In Section \ref{proofofthemaintheorem}, we establish a connection between Coleman maps and Kobayashi ranks, and show how this relationship enables us to study the growth of certain local modules. We conclude by combining these results to prove Theorem \ref{maintheorem}.

\subsection{Outlook}
Our proof of Theorem~\ref{maintheorem} follows closely the ideas of \cite[\S10]{kob03}. In \textit{loc. cit.}, the author presented a detailed study of the growth of modules of the form $\left(\Lambda/\langle f\rangle\right)_{\Gamma_n}$, where $f$ is a non-zero element of $\Lambda$, as $n\rightarrow\infty$. In the context of the current article, several Iwasawa modules have rank 2, and we are led to study the growth of modules of the form $\left(\Lambda^{\oplus 2}/\langle u_1,u_2\rangle\right)_{\Gamma_n}$, where $u_1,u_2$ are linearly independent elements of $\Lambda^{\oplus 2}$. It should be possible to extend our results to modules arising from $\Lambda^{\oplus d}$, $d\ge2$. This would potentially allow us to remove some of the hypotheses imposed in \cite{LL22} to study elliptic curves defined over more general number fields. 

\vspace{2mm}

The hypothesis $p\nmid h_K$ ensures that there is only one prime lying above $p$ inside the anticyclotomic tower. It should be possible to extend some of the results of the present article without this hypothesis, though not necessarily in a straightforward manner. One would have to construct local Coleman maps for each prime lying above $p$ and consider the image of the direct sum of these maps.

\vspace{2mm}

In a different vein, it would be interesting to study similar results in the context of the present article without assuming $a_p(E)=0$. This would require an appropriate extension of the results of \cite{bko21,bko23} on local points. 

\vspace{2mm}

Finally, we remark that a unified approach to study anticyclotomic Iwasawa main conjectures has been introduced in \cite{BLV} that treats the $p$-split and $p$-inert cases simultaneously. The method introduced in this article is specifically tailored for the $p$-inert case. For example, we crucially make use of the $\ZZ_{p^2}$-module structure of the local points attached to the formal group of $E$ and carry out explicit calculations specifically for rank-two modules over $\Lambda$. It can be very interesting to develop an alternative approach in line with \cite{BLV} that would apply to the $p$-split and $p$-inert cases uniformly.

\subsection*{Acknowledgements} We thank Ashay Burungale, F{\i}rt{\i}na K{\"u}{\c c}{\"u}k and Meng Fai Lim for interesting discussions during the preparation of this article. We also thank the anonymous referees for many helpful comments and suggestions. The authors' research is supported by the NSERC Discovery Grants Program RGPIN-2020-04259 and RGPAS-2020-00096. EI is also partially supported by a postdoctoral fellowship from the Fields Institute.

\section{Plus and minus objects}\label{localpoints}

\subsection{Local points of Burungale--Kobayashi--Ota}
For $0 \leq n \leq \infty$, let $k_n$ denote the localization of $K_n$ at the unique prime above $p$ (the uniqueness of this prime is a consequence of $p\nmid h_K$). Let $\fram_n$ denote the maximal ideal of the integer ring of $k_n$. For simplicity, set $k:=k_0$, and let $\ccO$ denote the integer ring of $k$. When $n <\infty$, we identify the Galois group $\gal(k_n/k)$ with $G_n$, and similarly $\gal(k_\infty/k)$ with $\Gamma$. 

\vspace{2mm}
 
 Let $\hatE$ be the formal group of the minimal model of $E$ over $\ccO$.  Let $\lambda:=\log_{\hatE}$ denote the logarithm of $\hatE$ which is normalized by $\lambda'(0)=1$. 

 \vspace{2mm}

 Let $\Xi$ denote the set of finite characters of $\Gamma$. For $\chi\in\Xi$, we say that $\chi$ is of order $p^n$ if it factors through $G_n$, but not $G_{n-1}$. Define
\begin{align*}
    \Xi^+&=\left\{\chi\in\Xi\mid \text{the order of }\chi \text{ is a positive even power of } p\right\};\\
    \Xi^-&=\left\{\chi\in\Xi\mid \text{the order of }\chi \text{ is an odd power of } p\right\}\cup\{\mathbf{1}\}.
\end{align*}
Here, $\mathbf{1}$ is the trivial character.

\vspace{2mm}

For $n\geq 0$, let $\Xi_n^\pm$ denote the set of $\chi\in \Xi^\pm$ factoring through $\gal(k_n/k)$. For $\chi\in \Xi^\pm_n$ and $x\in\hatE(\fram_n)$, let
\begin{equation*}
    \lambda_\chi(x):=\frac{1}{p^n}\sum_{\sigma\in \gal(k_n/k)}\chi^{-1}(\sigma)\lambda(x)^\sigma.
\end{equation*}
Define
\begin{equation*}
    \hatE^\pm(\fram_n)= \left\{x\in \hatE(\fram_n) \mid \lambda_\chi(x)=0 \text{ for all } \chi\in \Xi_n^\mp\right\}.
\end{equation*}

\begin{lemma}\label{tracecompability}
    For $n\geq 0$, there exist $c_n^+,c_n^-\in \hatE(\fram_n)$ such that:
    \begin{enumerate}[label=(\roman*)]
        \item If $(-1)^{n+1}=\pm1$ and $n\ge1$, then
        \begin{equation*}
            \tr_{n+1/n}c_{n+1}^\pm= c_{n-1}^\pm ,\quad c_n^\pm=c_{n-1}^\pm,
        \end{equation*}
where $\tr_{n+1/n}: \hatE(\fram_{n+1})\longrightarrow \hatE(\fram_n)$ is the trace map.
 \item We have $c_n^\pm\in\hatE^\pm(\fram_n)$.
    \end{enumerate}
   
\end{lemma}
\begin{proof}
    This is \cite[Lemma 2.6]{bko23} when $p>3$. The case where $p=3$ follows from the same proof after replacing \cite[Theorem 2.1]{bko21} by \cite[Theorem 4.7]{YZ24}.
\end{proof}

\vspace{2mm}

Define the Iwasawa algebra $\Lambda_\ccO:=\ccO[[\Gamma]]$ and set $\Lambda_{\ccO,n}:= \Lambda_\ccO\big/(\omega_n)=\ccO[G_n]$, where $\omega_n=(1+X)^{p^n}-1$ for $n\geq 0$.

\begin{theorem}\label{generatingtheorem}
    Let $n\geq 0$ be an integer.
    \begin{enumerate}[label=(\roman*)]
        \item As $\Lambda_{\ccO,n}$-modules, we have $\hatE(\fram_n)=\hatE^+(\fram_n)\oplus \hatE^-(\fram_n)$.
        \item $\hatE^\pm(\fram_n)$ is generated by $c_n^\pm$ as a $\Lambda_{\ccO,n}$-module.
    \end{enumerate}
\end{theorem}
\begin{proof}
    This is \cite[Theorem 2.7]{bko23} when $p>3$. Once again, we replace the input of \cite[Theorem 2.1]{bko21} by \cite[Theorem 4.7]{YZ24} in the case where $p=3$.
\end{proof}

\vspace{2mm}

\begin{remark}\label{rk:pmGroups}

\begin{itemize}
 \item[(i)] Note that $\Xi_0^+=\Xi_1^+=\emptyset$. Theorem~\ref{generatingtheorem} tells us that $c_0^-$ is an $\ccO$-basis of $\hatE(\fram_0)=\hatE^-(\fram_0)$, whereas $\hatE^+(\fram_0)=\hatE^+(\fram_1)=\{0\}$. 
      \item[(ii)] Let $n\ge 0$ be an integer and  $x\in\hatE(\fram_n)$. We can write uniquely $$\lambda(x)=\sum_{i=0}^nx_i,$$ where $x_i$ is an element of the kernel of the trace map $\mathrm{Tr}_{i/i-1}:k_i\rightarrow k_{i-1}$ for $i\ge1$ and $x_0\in k_{0}$. If $\chi\in\Xi$ is a character of order $p^i$, then $\lambda_\chi(x)=0$ if and only if $x_i=0$. This implies that
      \[
      \hatE^-(\fram_n)=\left\{x\in\hatE(\fram_n):\mathrm{Tr}_{n/i}\lambda(x)\in k_{i-1}, i\in 2\ZZ\cap[0,n]\right\}.
      \]
      In other words, $\hatE^-(\fram_n)$ can be described in the same manner as the cyclotomic counterpart given in \cite[Definition 8.16]{kob03}. However, this does not apply to $\hatE^+(\fram_n)$. In fact, both plus and minus subgroups defined in \textit{loc. cit.} contain $\hatE(p\Zp)$, whereas the group $\hatE^+(\fram_n)$ defined here does not contain any non-zero elements of $\hatE(\fram_0)$. 
       \item[(iii)]The plus and minus subgroups studied here are defined in the same way as those studied in \cite{rub87loc,AH,bko21,bko23,BLV,bbl24}. The necessity of divergence from the cyclotomic setting is related to the presence of an exceptional zero in one of the two $p$-adic $L$-functions. This phenomenon is discussed in detail in \cite{BLV}, where  for further details (see in particular the formulation of the Iwasawa main conjecture for the "exceptional case").
\end{itemize}
\end{remark}

\subsection{Local Tate pairing}
Let $T$ be the $p$-adic Tate module of $E$.  We regard $\hatE(\fram_n) \subseteq H^1(k_n , T )$ via the Kummer
map. Let 
 \begin{equation}\label{kummerpairing}
     \langle-,-\rangle_{n}: \hatE(\fram_n)\times  H^1(k_n,T) \xrightarrow{\;\; \cup\;\;} H^2(k_n,\ccO(1))\xrightarrow{\;\;\sim\;\;} \ccO
 \end{equation}
 denote cup-product pairing.
 
\vspace{2mm}

\begin{definition}\label{def:pairing}
    For $c\in\hatE(\fram_n)$, we define the $\ccO[G_n]$-morphism
 \begin{align*}
     P_{n,c}: H^1(k_n,T)&\longrightarrow \ccO[G_n]\\ 
      z&\mapsto \sum_{\sigma \in G_n} \langle c^\sigma,z\rangle_{n}\cdot\sigma.
 \end{align*}
\end{definition}
 
 Furthermore, as in \cite[Lemma 8.15]{kob03}, the following lemma holds.
 
 \begin{lemma}\label{compatibility}
      The following diagram commutes:
     \begin{center}
  \begin{tikzpicture}
   \node (Q1) at (0,2) {$H^1(k_n,T)$};
     \node (Q2) at (4,2) {$\Lambda_{\ccO,n}$};
      \node (Q3) at (0,0) {$H^1(k_{n-1},T)$};
    \node (Q4) at (4,0) {$\Lambda_{\ccO,n-1}$};

    \draw[->] (Q1)--node [above] {$P_{n,c}$}(Q2);
      \draw[->] (Q1)--node [left] {$\cor_{k_n/k_{n-1}}$}(Q3);
    \draw[->] (Q3)--node [above] {$P_{n-1,\tr_{n/n-1}(c)}$}(Q4);
    \draw[->] (Q2)--node [right] {$\pr$}(Q4);
           \end{tikzpicture}
\end{center}
where the vertical maps are the corestriction and the projection.
 
 \end{lemma}

 \vspace{2mm}

 \begin{lemma}\label{plusminusinjective}
     The Kummer map induces an injection $\hatE^\pm(\fram_n)\otimes\bbQ_p/\bbZ_p \hookrightarrow H^1(k_n,E[p^\infty])$.
 \end{lemma}
\begin{proof}
See \cite[Lemma~8.17]{kob03}.
\end{proof}

\vspace{2mm}

\begin{definition}
     For $\bullet\in\{+,-\}$, we define $H^1_\bullet(k_n,T)$ as the orthogonal complement of $\hatE^\pm(\fram_n)\otimes\bbQ_p/\bbZ_p\subseteq H^1(k_n,E[p^\infty])$ with respect to the local Tate pairing
       \begin{equation*}
      H^1(k_n,T)\times  H^1(k_n,E[p^\infty]) \longrightarrow k/\ccO.
      \end{equation*}
\end{definition}

\vspace{2mm}

\begin{prop}\label{kernelofthepairing}
    Define $P_n^\pm:=P_{n,c_n^\pm}$. Then the kernel of $P_n^\pm$ is $H^1_\pm(k_n,T)$.
\end{prop}
\begin{proof}
The proof is the same as that of \cite[Proposition 8.18]{kob03}.
\end{proof}

\vspace{2mm}

For $n\geq 0$, recall that $\omega_n=(1+X)^{p^n}-1=\displaystyle\prod_{0\le m\le n}\Phi_m$, where $\Phi_{m}$ is the $p^m$-th cyclotomic polynomial in $1+X$. We put
\begin{align}\label{cyclotomicpolynomials}
      \omega^+_n&:= X \prod_{\substack{1\leq m\leq n,\\ m: \text{ even}}} \Phi_{m} ,\\
    \omega_n^-&:=X \prod_{\substack{1\leq m\leq n,\;\\ m: \text{ odd}}} \Phi_{m},
\end{align}
and $\tomg_n^\pm:=\omega_n^\pm/X$ for $n\ge1$, and $\omega^+_0 := X$, $\tomg^-_0:=1$.\footnote{We have followed the convention introduced in \cite{kob03} since our proof of Theorem~\ref{maintheorem} is based on a generalization of the methods introduced therein. We caution the reader that different conventions exist in the literature. For example, our notation is slightly different from the one introduced in \cite{BLV}.} 

\vspace{2mm}

\begin{prop}\label{imageofthepairing}
    The image of $P_n^+$ is contained in $\omega_n^-\Lambda_{\ccO,n} $, whereas the image of $P_n^-$ is contained in $\tomg_n^+\Lambda_{\ccO,n}$.
\end{prop}
\begin{proof}
 The proof is similar to \cite[Proposition 8.19]{kob03}. We illustrate this for $P_n^+$. Let $\chi$ be a character of $G_n$ that belongs to $\Xi_n^-$. Then, $\chi$ sends $\gamma=X+1$ to $1$ or a $p^m$-th primitive root of unity for some odd non-negative integer $m\le n$. Since 
$c_n^+\in\hatE^+(\fram_n)$,
 \[
 P_{n}^+(z)(\chi)=\sum_{\sigma\in G_n}\langle \chi(\sigma)(c_n^+)^\sigma,z\rangle_n=0.
 \]
 Therefore, as an element of $\Lambda_{\ccO,n}$, $P_n^+(z)$ is divisible by $\Phi_m$ for all non-negative odd integers $m\le n$. Therefore, 
 $$P_{n}^+(z)\in X\prod_{\substack{1\leq m\leq n,\;\\ m: \text{ odd}}} \Phi_{m}\Lambda_{\ccO,n}=\omega_n^-\Lambda_{\ccO,n}.$$
 Similarly, we have
 $$P_{n}^-(z)\in \prod_{\substack{1\leq m\leq n,\;\\ m: \text{ even}}} \Phi_{m}\Lambda_{\ccO,n}=\tomg_n^+\Lambda_{\ccO,n}.$$
 \end{proof}
Note that $\omega_n^-\Lambda_{\ccO,n}\subset \tomg_n^-\Lambda_{\ccO,n}$. In particular, Proposition~\ref{imageofthepairing} tells us that the image of $P_n^+$ is contained in $\tomg_n^-\Lambda_{\ccO,n}$. We follow the strategy of \cite[\S8]{kob03} to use this weaker congruence to define the corresponding Coleman map in the following section.

\subsection{Signed Coleman maps}  
It is clear from the definitions that $\omega_n=\tomg_n^\mp\omega_n^\pm$. Thus, the multiplication by $\tomg_n^\mp$ induces an isomorphism
\begin{equation}\label{signediwasawaalgebra}
\Lambda_{\ccO,n}^\pm:=\Lambda_{\ccO}\big/\left(\omega_n^\pm \right)\xrightarrow{\;\;\sim\;\;} \tomg_n^\mp \Lambda_{\ccO,n}.
\end{equation}

\vspace{2mm}

\begin{lemma}
   There exist unique morphisms $\col_n^\pm$ such that the following diagram commutes:
\begin{center}
  \begin{tikzpicture}
   \node (Q1) at (0,2) {$H^1(k_n,T)$};
     \node (Q2) at (4,2) {$\Lambda_{\ccO,n}^\pm$};
      \node (Q3) at (0,0) {$\displaystyle\frac{H^1(k_n,T)}{H^1_\pm(k_n,T)}$};
    \node (Q4) at (4,0) {$\Lambda_{\ccO,n}$};

    \draw[->] (Q1)--node [above] {$\col_n^\pm$}(Q2);
      \draw[->] (Q1)--(Q3);
    \draw[->] (Q3)--node [above] {$P_{n}^\pm$}(Q4);
    \draw[right hook->] (Q2)--node [right] {$\times \tomg_n^\mp$}(Q4);
               \end{tikzpicture}
\end{center}
    
    \end{lemma}

\begin{proof}
    This follows from Propositions \ref{kernelofthepairing} and \ref{imageofthepairing}, together with \eqref{signediwasawaalgebra}.
\end{proof}

\vspace{2mm}

\begin{prop}
    The following diagram commutes:
       \begin{center}
  \begin{tikzpicture}
   \node (Q1) at (0,2) {$H^1(k_{n+1},T)$};
     \node (Q2) at (4,2) {$\Lambda_{\ccO,n+1}^\pm$};
      \node (Q3) at (0,0) {$H^1(k_n,T)$};
    \node (Q4) at (4,0) {$\Lambda_{\ccO,n}^\pm$};

    \draw[->] (Q1)--node [above] {$\col_{n+1}^\pm$}(Q2);
      \draw[->] (Q1)--(Q3);
    \draw[->] (Q3)--node [above] {$\col_n^\pm$}(Q4);
    \draw[->] (Q2)--(Q4);
           \end{tikzpicture}
\end{center}
where the vertical maps are the corestriction and the projection.
\end{prop}
\begin{proof}
    The proof is the same as \cite[Proposition~8.21]{kob03}.
\end{proof}
\vspace{2mm}

Observe that 
\begin{equation*}
    \varprojlim_n \Lambda_{\ccO,n}^\pm=  \varprojlim_n \Lambda_{\ccO}\big/\left(\omega_n^\pm\right) = \Lambda_\ccO.
\end{equation*}
This allows us to give the following definition.

\begin{definition}
    We define the plus and minus Coleman maps
    \begin{equation*}
        \col^\pm :  H^1_{\iw}(k_\infty,T) \longrightarrow \Lambda_\ccO
    \end{equation*}
    as the inverse limits of $\col_n^\pm: H^1(k_n,T)\longrightarrow \Lambda_{\ccO,n}^\pm$.
\end{definition}

\subsection{Image of the signed Coleman maps}

\vspace{2mm}
\begin{lemma}\label{corestrictionsurjective}
    The corestriction map $H^1(k_m,T)\longrightarrow H^1(k_n,T)$ is surjective for all $m\geq n$.
\end{lemma}
\begin{proof}
   The proof is similar to \cite[Lemma 2.1]{LL22}.
  \end{proof}

\vspace{2mm}

Let $I_{\ccO,n}$ be the augmentation ideal of $\Lambda_{\ccO,n}=\ccO[G_n]$:
\begin{equation*}
    I_{\ccO,n}=\ker\left( \ccO[G_n]\longrightarrow \ccO\right).
\end{equation*}
We denote $I_\ccO:=\varprojlim_n I_{\ccO,n}=X\Lambda_\ccO$ and $I_{\ccO,n}^\pm=I_{\ccO,n}\cap \Lambda^\pm_{\ccO,n}=X\Lambda_{\ccO,n}^\pm$, where the last equality follows from \eqref{signediwasawaalgebra}.

\vspace{2mm}

\begin{prop}\label{surjectivityofcolemanmaps}
     The image of $\col^-$ (resp. $\col^-_n$) is equal to $\Lambda_\ccO$ (resp. $\Lambda_{\ccO,n}^-$), whereas the image of $\col^+$ (resp. $\col_n^+$) is given by $I_\ccO$ (resp. $I_{\ccO,n}^+$).
\end{prop}
\begin{proof}
      By Lemma \ref{corestrictionsurjective} and Nakayama’s lemma, it is enough to show that $\image\left(\col^-_0\right)=\Lambda_{\ccO,0}^-=
      \ccO$ and $\image\left(\col^+_2\right)=I_{\ccO,2}^+$. Note that $\hatE(\fram_0)$ is generated by $c_0^-$ as an $\ccO$-module, so the proof of \cite[Proposition 8.23]{kob03} can be extended to show that $\im\left(\col^-_0\right)=\ccO$. 

We have the following $\ccO$-isomorphism $$I_{\ccO,2}^+=X\Lambda_{\ccO,2}^+=\frac{X\Lambda_\ccO}{(X\Phi_2)}\simeq\ccO[\zeta_{p^2}],$$
where $\zeta_{p^2}$ is a primitive $p^2$-th root of unity, and $Xf(X)$ is sent to $f(\zeta_{p^2}-1)$ under this isomorphism.
       Since $\left\{\zeta_{p^2}^i:0\le i\le \phi(p^2)-1\right\}$ is an $\ccO$-basis of $\ccO[\zeta_{p^2}]$, we see that 
            $\left\{X(1+X)^i:0\le i\le\phi(p^2)-1\right\}$ is an $\ccO$-basis of $I_{\ccO,2}^+$.

      Let $\sigma$ be the image of our chosen topological generator $\gamma=1+X$ of $\Gamma$ in $G_2$. In particular, it is a generator of the cyclic group $G_2$, and $\gal(k_2/k_1)=\langle\sigma^p\rangle$. Since $c_2^+\in\hatE^+(\fram_2)$, for all integer $j$, we have
      $$ \sum_{i=0}^{p-1}\left(c_2^+\right)^{\sigma^{j+ip}}=\tr_{2/1}\left(c_2^+\right)^{\sigma^j}=0.$$
Furthermore, Theorem~\ref{generatingtheorem} tells us that
\begin{equation}
\label{eq:decomp}
\hatE(\fram_2)=\hatE^+(\fram_2)\oplus\hatE^-(\fram_2)=\hatE^+(\fram_2)\oplus\hatE(\fram_1),    
\end{equation}
and $\hatE^+(\fram_2)=\ccO[G_2]\cdot c_2^+$.
      Therefore,  $\left\{\left(c_2^+\right)^{\sigma^i}:0\le i\le\phi(p^2)-1\right\}$ is an $\ccO$-basis of $\hatE^+(\fram_2)$. This gives the following $\ccO[G_2]$-isomorphism
      \begin{equation}
      \hatE^+(\fram_2)\simeq I_{\ccO,2}^+,
      \label{eq:augment}
      \end{equation}
      where $c_2^+$ is sent to $X$.
          Consequently, the morphism $\col_2^+$ can be described as 
      \begin{equation*}
          H^1(k_2,T)\longrightarrow \Hom\left(\hatE(\fram_2),\ccO[G_2]\right) \longrightarrow \Hom\left(\hatE^+(\fram_2),\ccO[G_2]\right) \simeq I_{\ccO,2}^+,
      \end{equation*}
      where the first map is defined as
      \[
      z\mapsto\left(Q\mapsto P_{2,Q}(z)\right)
      \]
      (see Definition~\ref{def:pairing}), the second map is the projection by the decomposition \eqref{eq:decomp}, and the last isomorphism is induced from \eqref{eq:augment}. The first arrow is surjective since its dual is the injection $\hatE(\fram_2)\otimes\bbQ_p/\bbZ_p \longrightarrow H^1(k_2,E[p^\infty])$. It is clear from the definition that the second arrow is surjective. Therefore, the image of $\col_2^+$ is $I_{\ccO,2}^+$, as desired.
\end{proof}

\subsection{Signed Selmer groups}\label{signedselmergroups}  Recall that $K_\infty$ is the anticyclotomic $\bbZ_p$-extension of $K$, and that $K_n \subset K_\infty$ denotes the unique subextension such that $[K_n : K] = p^n$, for an integer $n \geq 0$.

For a rational prime $\ell$, let
\begin{equation*}
    K_{n,\ell}:=K_n\otimes_\bbQ \bbQ_\ell,\;\; H^1(K_{n,\ell}, E[p^\infty]):=\bigoplus_{\lambda\mid \ell} H^1(K_{n,\lambda}, E[p^\infty]),
\end{equation*}
where the direct sum runs over all primes of $K_n$ above $\ell$. We have the natural restriction map
\begin{equation*}
    \res_\ell: H^1(K_{n}, E[p^\infty]) \longrightarrow H^1(K_{n,\ell}, E[p^\infty]).
\end{equation*}

Let $H^1_{\bff}(K_{n,\ell}, E[p^\infty])\subset H^1(K_{n,\ell}, E[p^\infty])$ for the Bloch--Kato subgroup. The singular quotient is given by
\begin{equation*}
    H^1_{/\bff}(K_{n,\ell}, X):=\dfrac{H^1(K_{n,\ell}, E[p^\infty])}{H^1_{\bff}(K_{n,\ell}, E[p^\infty])}.
\end{equation*}

\vspace{2mm}

\begin{definition}\label{def:pmSel}
   For $\bullet\in\{+,-\}$, we define the signed Selmer group of $E$ over $K_n$ by
    \begin{equation*}
        \Sel^\bullet(E/K_n):=\ker\left(H^1(K_n,E[p^\infty])\longrightarrow \prod_{\ell\nmid p} H^1_{/\bff}(K_{m,\ell}, E[p^\infty])\times\dfrac{H^1(k_{n}, E[p^\infty])}{\hatE^\bullet(\fram_n)\otimes\bbQ_p/\bbZ_p} \right).
    \end{equation*}
    Further, define $\Sel^\bullet(E/K_\infty)=\varinjlim_n \Sel^\bullet(E/K_n)$.
\end{definition}

\vspace{2mm}

If $\Sel_{p^\infty}(E/K_\infty)$ denotes the classical $p^\infty$-Selmer group, then $\Sel^\bullet(E/K_\infty)\subset \Sel_{p^\infty}(E/K_\infty)$. It follows from \cite[Theorem 4.5]{man71} that $\Sel_{p^\infty}(E/K_\infty)$ is cofinitely generated over $\Lambda$. Hence, so is $\Sel^\bullet(E/K_\infty)$. We consider the following hypothesis.

\vspace{2mm}

\begin{conj}\label{cotorsionconjecture}
    For $\bullet\in\{+,-\}$, the signed Selmer groups $\Sel^\bullet(E/K_\infty)$ is cotorsion over $\Lambda$.
\end{conj}

\begin{remark}\label{cotorsionresult}
   Let $N$ denote the conductor of $E$, and $p\geq 5$. Assume that $(D_K,pN)=1$, where $D_K$ is the discriminant of $K$. Write $N=N^+N^-$, where $N^+$ (resp. $N^-$) is divisible only by primes that are split (resp. inert) in $K$. In \cite{BLV,bbl24}, it has been proved that Hypothesis \ref{cotorsionconjecture} is valid under the following hypotheses:
   \begin{enumerate}[label=(\roman*)]

       \item $N^-$ is a square-free product of odd number of primes.
       \item If $p=5$, the residual representation $\overline{\rho}_E\left( G_ {\bbQ(\mu_{p^\infty})}\right)$ contains a conjugate of $\SL_2(\bbF_p)$. If $p>5$, the $G_\bbQ$-representation $\overline{\rho}_E$ is irreducible.
       \item $\overline{\rho}_E$ is ramified at the primes $\ell$ that satisfy one of the following conditions:
       \begin{itemize}
           \item $\ell \mid N^-$ with $\ell^2 \equiv 1 \mod p$,
           \item $\ell \mid N^+$.
       \end{itemize}
   \end{enumerate}
\end{remark}

For the remainder of the article, we assume that Hypothesis \ref{cotorsionconjecture} holds.

\begin{definition}
  For $\bullet\in\{+,-\}$, we write $\mu_\bullet$ and $\lambda_\bullet$ for the $\mu$- and $\lambda$-invariants of the torsion $\Lambda$-module $\Sel^\bullet(E/K_\infty)^\vee$.
  \end{definition}

\vspace{2mm}
Let $\Sigma$ denote a fixed finite set of primes of $K$ containing $p$, the ramified primes of $K/\bbQ$, the archimedean place, and all the bad reduction primes of $E$. Write $K_\Sigma$ for the maximal algebraic extension of $K$ which is unramified outside $\Sigma$. For any (possibly infinite) extension $K \subseteq L \subseteq K_\Sigma$, write $G_\Sigma(L) = \gal(K_\Sigma/L)$. 
For $i\in\{1,2\}$, we define $H^i_{\iw,\Sigma}(K_\infty,T)=\displaystyle\varprojlim_n H^i(G_\Sigma(K_n),T)$, where the transition maps are given by the corestriction maps. Note that $H^i_{\iw,\Sigma}(K_\infty,T)$ is independent of the choice of $\Sigma$ (see \cite[Corollary B.3.5]{Ru00}, \cite[Lemma 5.3.1]{MR04} and \cite[Proposition 7.1]{kob03}). Since the set $\Sigma$ is fixed, we will drop the subscript $\Sigma$ from the notation for simplicity and simply write $H^i_{\iw}(K_\infty,T)$.

 \vspace{2mm}

We conclude this section with the following statement on the structure of $H^i_\iw(K_\infty,T)$.

\vspace{2mm}

\begin{prop}\label{consequenceofcotorsionness}
    If Hypothesis~\ref{cotorsionconjecture} holds,  the following statements are valid:
    \begin{enumerate}[label=(\roman*)]
        \item $H^1_\iw(K_\infty,T)$ is a free $\Lambda$-module of rank $2$.
        \item $H^2_\iw(K_\infty,T)$ is a torsion $\Lambda$-module.
    \end{enumerate}
\end{prop}
\begin{proof}
  See \cite[Proposition 2.9]{LL22}.
\end{proof}

\section{Kobayashi Ranks}\label{kobayashiranks}
\subsection{Definition and basic properties}
Following \cite[\S 10]{kob03}, we define the Kobayashi ranks as follows.

\begin{definition}
    Let $(M_n)_{n\geq 1}$ be a projective system of finitely generated $\bbZ_p$-modules. Given an integer $n\ge1$, if $\pi_n:M_n\longrightarrow M_{n-1}$ has finite kernel and cokernel, we  define
\begin{equation*}
    \nabla M_n:=\length_{\bbZ_p}(\ker\pi_n)-\length_{\bbZ_p}(\coker\; \pi_n)+\dim_{\bbQ_p}M_{n-1}\otimes \bbQ_p.
\end{equation*}
\end{definition}



\begin{lemma}\label{kobayashirankzero}Let $(M'_n)_{n\ge1}$, $(M_n)_{n\ge1}$ and $ (M''_n)_{n\ge1}$ be projective systems of finitely generated $\Zp$-modules.
\begin{enumerate}[label=(\roman*)]
    \item Suppose that for all $n\ge1$, there is an exact sequence
    \begin{equation*}
        0\longrightarrow (M'_n)\longrightarrow (M_n)\longrightarrow (M''_n)\longrightarrow 0.
    \end{equation*}
    If two of $\nabla M_n, \nabla M'_n,\nabla M''_n$ are defined, then the other is also defined, in which case
    \begin{equation*}
        \nabla M_n=\nabla M'_n+\nabla M''_n.
    \end{equation*}
    \item  Suppose that $M_n$ are constant. If $M_n$ is finite or the transition map $M_n\longrightarrow M_{n-1}$ is given by the multiplication map by $p$, then $\nabla M_n=0$.
\end{enumerate}    
\end{lemma}
\begin{proof}
    See \cite[Lemma 10.4]{kob03}.
\end{proof}

\vspace{2mm}

\begin{lemma}\label{kobayashirankandcharacteristicpolynomial}
 Let $f\in \Lambda$ be a non-zero element of $\Lambda$. Let $(N_n)_{n\ge1}$ be a projective system given by $N_n= \Lambda\big/ (f,\omega_{n})$, where the connecting maps are natural projections.

\begin{enumerate}[label=(\roman*)]
    \item Suppose that $\Phi_n\nmid f$. Then $\nabla N_n$ is defined and is equal to $\ord_{\epsilon_n}f(\epsilon_n)$, where $\epsilon_n=\zeta_{p^n}-1$.
   
     \item Let $M$ be a finitely generated torsion $\Lambda$-module with the characteristic polynomial $f$, and let $M_n=M/\omega_nM$. Consider the natural projective system $(M_n)_{n\geq 1}$. Then, for $n\gg 0$, $\nabla M_n$ is defined and
    \begin{equation*}
        \nabla M_{n}= \ord_{\epsilon_n}f(\epsilon_n)=\lambda(M)+\phi(p^n)\mu(M),
    \end{equation*}
where $\lambda(M)$ and $\mu(M)$ are the Iwasawa invariants of $M$ and $\phi$ is the Euler totient function.
\end{enumerate}

\end{lemma}

\begin{proof}
    See \cite[Lemma~10.5]{kob03}.
\end{proof}

\subsection{Kobayashi ranks of modules arising from $2\times 2$ matrices}
We derive a formula for the Kobayashi ranks of modules arising from certain types of $2\times 2$ matrices defined over $\Lambda$. We begin with the following definition.

\begin{definition}\label{def:app}
    Let $A=\begin{bmatrix}
        a&c\\b&d
    \end{bmatrix}\in \mat$.
   \begin{enumerate}[label=(\roman*)]
       \item We write $\langle A\rangle\subseteq\Ltwo$ for the $\Lambda$-module generated by the columns of $A$. Similarly, given two such matrices $A$ and $B$, we write $\langle A,B\rangle\subseteq\Ltwo$ for the $\Lambda$-module generated by the columns of $A$ and $B$.
    \item For $n\ge0$, let $\Lambda_n=\Lambda/(\omega_n)$ and write $\langle A\rangle_n$ for the $\Lambda_n$-module generated by the columns of $A$ modulo $\omega_n$ inside $\Ltwo_n$. Furthermore, we write $\An=\Ltwo_n/\langle A\rangle_n=\Ltwo/\langle \omega_n I_2,A\rangle$, where $I_2$ denotes the $2\times 2$ identity matrix. 
    \item We say that $A$ is \textbf{special} relative to an integer $n$ if it satisfies the following property:
    \[
    \Phi_m\mid \det(A),\; 0\le m\le n \; \Rightarrow \; \Phi_m \mid a,b\text{ or } \Phi_m \mid c,d.
    \]

   \end{enumerate} 
   
\end{definition}

The following lemma is a consequence of the Chinese remainder theorem.
\begin{lemma}\label{lem:CRT}
Let $n\ge1$ be an integer.
    \begin{itemize}
        \item[(i)] For all $A\in \mat$, we have $$A_{(n)}\otimes\Qp\simeq\bigoplus_{0\le m\le n}\Qp(\zeta_{p^m})^{\oplus 2}\big/\langle A(\epsilon_m)\rangle,$$ where $\langle A(\epsilon_m)\rangle$ denotes the image of the $\Qp(\zeta_{p^m})$-linear endomorphism on $\Qp(\zeta_{p^m})^{\oplus 2}$ defined by the matrix $A(\epsilon_m)$. 
        \item[(ii)] Let $\ccS$ be a finite set of non-negative integers. For each $i\in\ccS$, let $x_i\in \Qp(\zeta_{p^i})$ be any element. There exists a polynomial $F\in \Qp[X]$ such that $F(\epsilon_i)=x_i$ for all $i\in\ccS$.
    \end{itemize}    
\end{lemma}
\begin{proof}
Let $R=\Lambda\otimes\Qp$ and $M=(\Lambda^{\oplus2}/\langle A\rangle)\otimes\Qp$. Then, the Chinese remainder theorem for $R$-modules says that
\[
M/\omega_n M\simeq\bigoplus_{m=0}^n M/\Phi_m M.
\]
It is clear from the definition that $M/\omega_n M\simeq A_{(n)}\otimes\Qp$. Furthermore, the evaluation map $X\mapsto \epsilon_m$ induces the isomorphism $M/\Phi_m M\simeq \Qp(\zeta_{p^m})^{\oplus 2}\big/\langle A(\epsilon_m)\rangle$. Hence, assertion (i) follows.

Assertion (ii) follows from the isomorphism
\[
\Qp[X]\left/\left(\prod_{i\in\ccS}\omega_i\right)\right.\simeq\bigoplus_{i\in\ccS} \Qp[X]/(\Phi_i)\simeq\bigoplus_{i\in\ccS}\Qp(\zeta_{p^i}),
\]
which once again follows from the Chinese remainder theorem.
\end{proof}

We prove the following generalization of \cite[Lemma~10.5i)]{kob03}.
\begin{theorem}\label{thm:app}
Let $n\ge0$ be an integer. If $A\in\mat$ is special relative to $n$ and $\Phi_n\nmid\det(A)$, then $\nabla \An$ is defined and is equal to $\ord_{\epsilon_n}\left(\det A(\epsilon_n)\right)$. Here, the connecting map $\pi_n:\An\rightarrow A_{(n-1)}$ is the natural projection.
\end{theorem}

\begin{proof}
For each integer $m$ such that $0\le m\le n-1$, we write $i_m\in\{0,1,2\}$ for the number of columns of $A$ that are divisible by $\Phi_m$. Then, $A$ can be uniquely written as
\[
A=BD,
\]
where $B,D\in\mat$ such that $D$ is a diagonal matrix whose diagonal entries are square-free products of elements of a subset $\{\Phi_m:0\le m\le n-1\}$, and the number of times $\Phi_m$ appears in $\det D$ is precisely $i_m$.

Lemma~\ref{lem:CRT}(i) gives $$A_{(n-1)}\otimes\Qp\simeq\bigoplus_{0\le m\le n-1}\Qp(\zeta_{p^m})^{\oplus 2}\big/\langle A(\epsilon_m)\rangle.$$ The rank of the matrix $A(\epsilon_m)$ is equal to $2-i_m$ as $A$ is special relative to $n$. Therefore, the rank-nullity theorem tells us that
\[
    \dim_{\Qp(\zeta_{p^m})}\Qp(\zeta_{p^m})^{\oplus 2}\big/\langle A(\epsilon_m)\rangle=i_m.
    \]
    Hence,
    \[
    \dim_{\Qp}A_{(n-1)}\otimes\Qp=\sum_{m=0}^{n-1}\phi(p^m)i_m=\ord_{\epsilon_n}\left(\det D(\epsilon_n)\right).
    \]

    As the connecting map $\pi_n$ is surjective, it remains to show that the length of $\ker\pi_n$ is given by $\ord_{\epsilon_n}\left(\det B(\epsilon_n)\right)$. The isomorphism theorem tells us that
    \begin{equation}
    \ker\pi_n\simeq\frac{\langle \omega_{n-1}I_2,A\rangle}{\langle \omega_{n}I_2,A\rangle}\simeq\frac{\langle \omega_{n-1}I_2\rangle}{\langle \omega_{n}I_2,A\rangle\bigcap\langle \omega_{n-1}I_2\rangle}.
   \label{eq:kernel}     
    \end{equation}

    \noindent\textbf{\underline{Claim:}} $\langle \omega_{n}I_2,A\rangle\bigcap\langle \omega_{n-1}I_2\rangle=\langle\omega_n I_2,\omega_{n-1}B\rangle$.

    It is clear that $\langle\omega_n I_2,\omega_{n-1}B\rangle\subseteq\langle\omega_{n-1} I_2\rangle$. 
 Since the diagonal entries of the diagonal matrix $D$ are square-free products of elements of a subset of $\{\Phi_m:0\le m\le n-1\}$, we may write $\omega_{n-1}I_2=DD'$ for some diagonal matrix $D'$.
    
    Thus,
    \[
    \omega_{n-1}B=BDD'=AD',
    \]
    which implies that $\langle \omega_{n-1}B\rangle \subseteq\langle A\rangle$. Therefore, we deduce that 
    \[
    \langle\omega_n I_2,\omega_{n-1}B\rangle\subseteq\langle \omega_{n}I_2,A\rangle\bigcap\langle \omega_{n-1}I_2\rangle.
    \]

    We now prove the opposite inclusion. Let $v\in \langle \omega_{n}I_2,A\rangle\bigcap\langle \omega_{n-1}I_2\rangle$. We can regard $v$ as a column vector with entries in $\Lambda$ and write 
    \begin{equation}\label{eq:defn-v}
    v=A\begin{bmatrix}
        x\\ y
    \end{bmatrix}+\omega_nu,     
    \end{equation}
      where $x,y\in\Lambda$ and $u\in\Ltwo$. Since $v\in\langle \omega_{n-1}I_2\rangle$, we have
    \begin{equation}
        A\begin{bmatrix}
        x\\ y
    \end{bmatrix}=BD\begin{bmatrix}
        x\\ y
    \end{bmatrix}\equiv 0\mod\Phi_m,\quad 0\le m\le n-1.
    \label{eq:mod-Phi_m}
    \end{equation}
         We consider three cases.

          \noindent\textbf{\underline{Case 1:}}  If $i_m=0$, $A(\epsilon_m)$ is an invertible matrix as it is special relative to $n$. In this case, \eqref{eq:mod-Phi_m} implies that $x,y\equiv 0\mod\Phi_m$.
            
            \noindent\textbf{\underline{Case 2:}} If $i_m=1$, there exists a non-zero element $\begin{bmatrix}
                a\\b
            \end{bmatrix}\in\Ltwo$ such that
            \begin{equation}\label{eq:dichot}
            A\equiv\begin{bmatrix}
                a&0\\
                b&0
            \end{bmatrix}\mod \Phi_m\quad\text{or}\quad A\equiv\begin{bmatrix}
                0&a\\
                0&b
            \end{bmatrix}\mod \Phi_m
            \end{equation}
            Let's suppose the first congruence relation holds. Then \eqref{eq:mod-Phi_m} implies that $\Phi_m\mid x$. Furthermore, $D\equiv\begin{bmatrix}
                d_m&0\\
                0&0
            \end{bmatrix}\mod\Phi_m$ for some non-zero $d_m\in\Lambda$. Therefore, $D\begin{bmatrix}
                x\\ y
            \end{bmatrix}\equiv0\mod\Phi_m$. If the second congruence relation in \eqref{eq:dichot} holds, the same argument shows that $D\begin{bmatrix}
                x\\ y
            \end{bmatrix}\equiv0\mod\Phi_m$.

               \noindent\textbf{\underline{Case 3:}} If $i_m=2$, then $D\equiv0\mod\Phi_m$. In particular, $D\begin{bmatrix}
                x\\ y
            \end{bmatrix}\equiv0\mod\Phi_m$.

            In all three cases, we have $D\begin{bmatrix}
        x\\ y
    \end{bmatrix}\equiv 0\mod\Phi_m$. Therefore, we deduce that \[
    D\begin{bmatrix}
        x\\ y
    \end{bmatrix}=\omega_{n-1}\begin{bmatrix}
        x'\\y'
    \end{bmatrix}\]
    for some $x',y'\in\Lambda$. Combined with \eqref{eq:defn-v}, we have
    \[
    v=\omega_{n-1}B\begin{bmatrix}
        x'\\y'
    \end{bmatrix}+\omega_nu\in \langle\omega_nI_2,\omega_{n-1}B\rangle.
    \]
    This shows that $\langle \omega_{n}I_2,A\rangle\bigcap\langle \omega_{n-1}I_2\rangle\subseteq\langle\omega_n I_2,\omega_{n-1}B\rangle$, and our claim follows.

We can now rewrite \eqref{eq:kernel} as
    \[\ker\pi_n\simeq\frac{\langle\omega_{n-1}I_2\rangle}{\langle \omega_n I_2,\omega_{n-1}B\rangle}\simeq\frac{\Ltwo}{\langle \Phi_n I_2,B\rangle},\]
    where the last isomorphism is given by
    \begin{align*}
        \frac{\Ltwo}{\langle \Phi_n I_2,B\rangle}&\rightarrow\frac{\langle\omega_{n-1}I_2\rangle}{\langle \omega_n I_2,\omega_{n-1}B\rangle}\\
        x&\mapsto \omega_{n-1}x.
    \end{align*}
    As $\Phi_n\nmid\det A$ by assumption, it follows from \cite[Lemma~7.8]{ip06} that $\ker\pi_n$ is finite, with length equal to $\ord_{\epsilon_n}\left(\det B(\epsilon_n)\right)$, as desired.
\end{proof}

\section{Proof of Theorem \ref{maintheorem}}\label{proofofthemaintheorem}

The proof of Theorem~\ref{maintheorem} is divided into a number of steps. Following \cite[\S 10]{kob03}, we define for each integer $n\ge0$
\begin{align*}
    \ccY(E/K_n)&:=\coker \left( H^1(G_\Sigma(K_n),T) \longrightarrow\frac{H^1(k_n,T)}{E(k_n)\otimes\bbZ_p}\right).
\end{align*}
The key ingredient to studying the growth of $\rank_\bbZ E(K_n)$ and $\Sha(E/K_n )[p^\infty]$ is understanding $\nabla \ccY(E/K_n)$. Similar to \cite[\S10]{kob03}, we do so using fine Selmer groups and several auxiliary modules. 
\subsection{Fine Selmer groups}
We recall the definition of fine Selmer groups:
\begin{definition}
    For $0\leq n \leq \infty$, we define the fine Selmer group  
        \begin{equation*}
     \Sel^0(E/K_n):=\ker\left( \Sel_{p^\infty}(E/K_n)\longrightarrow H^1(k_n, E[p^\infty]) \right).
    \end{equation*}
    
\end{definition}

Equivalently, we have
 \begin{equation*}
     \Sel^0(E/K_n):=\ker\left( H^1(G_\Sigma(K_n),E[p^\infty])\longrightarrow \prod_{\fraq_n\in \Sigma(K_n)} H^1(K_{n,\fraq_n}, E[p^\infty]) \right),
    \end{equation*}
where $\Sigma(K_n)$ denotes the set of primes of $K_n$ lying above $\Sigma$. The Pontryagin duals of $\Sel_{p^\infty}(E/K_n)$, $\Sel^\pm(E/K_n)$ and  $ \Sel^0(E/K_n)$ are denoted by $\ccX(E/K_n)$, $\ccX^\pm(E/K_n)$ and $\ccX^0(E/K_n)$, respectively.

\vspace{2mm}

\begin{lemma}\label{controltheorem}
    The natural restriction map
    \begin{equation*}
         \Sel^0(E/K_n) \longrightarrow  \Sel^0(E/K_\infty)^{\Gamma_n}
    \end{equation*}
    is injective and has finite cokernel. Furthermore, the cardinality of the cokernel stabilizes as $n\rightarrow\infty$.
\end{lemma}
\begin{proof}
The injectivity follows from the fact that $E(K_\infty)[p^\infty]=0$ and the inflation-restriction exact sequence. The assertions on the cokernel is a special case of \cite[Theorem 3.3]{Lim20} (see also \cite[Proposition~4.1]{Wut04}).
\end{proof}

\begin{lemma}\label{kobayashirankofX0}
Assume that Hypothesis~\ref{cotorsionconjecture} holds.   Let $n\geq 0$ be an integer.
    \begin{enumerate}[label=(\roman*)]
        \item We have the following short exact sequence:
        \begin{equation*}
            0 \longrightarrow \ccY(E/K_n) \longrightarrow \ccX(E/K_n) \longrightarrow \ccX^0(E/K_n) \longrightarrow 0.
        \end{equation*}

\vspace{1.5mm}
        
        \item  For $n\gg 0$, $\nabla \ccX^0(E/K_n)$ is defined and satisfies the equality
        \begin{equation*}
            \nabla \ccX^0(E/K_n)= \nabla \ccX^0(E/K_\infty)_{\Gamma_n}.
        \end{equation*}
    \end{enumerate}
\end{lemma}
\begin{proof}
   By \cite[Proposition A.3.2]{Per00}, we have the following Poitou--Tate exact sequence
            \begin{multline}\label{eqn:poitoutateexactsequence}
       H^1(G_\Sigma(K_n),T) \longrightarrow    \dfrac{H^1(k_n, T)}{E(k_n)\otimes \Zp} \longrightarrow \ccX(E/K_n) \longrightarrow \ccX^0(E/K_n) \longrightarrow 0.
\end{multline}
   Assertion (i) follows from the exact sequence \eqref{eqn:poitoutateexactsequence} and the definition of $\ccY(E/K_n)$ (see also \cite[(7.18) and (10.35)]{kob03}).

    \vspace{2mm}

As $\ccX^0(E/K_\infty)$ is a quotient of $\ccX^\pm(E/K_\infty)$, which are assumed to be $\Lambda$-torsion, it follows that $\ccX^0(E/K_\infty)$ is also $\Lambda$-torsion. By Lemma \ref{controltheorem}, we see that the kernel and cokernel of the natural map
    \begin{equation*}
        \ccX^0(E/K_\infty)_{\Gamma_n} \longrightarrow \ccX^0(E/K_n)
    \end{equation*}
    are finite and independent of $n$. Combining this with Lemma \ref{kobayashirankzero}(ii), the second assertion follows.
\end{proof}

\subsection{Calculating certain Kobayashi ranks via special matrices}
We introduce auxiliary modules (denoted by $M_n$) in preparation for the calculation of  $\nabla\ccY(E/K_n)$.
The composition
\begin{equation}\label{lambdahomomorphism}
    H^1_{\iw}(K_\infty,T) \xrightarrow{\;\loc_p\;} H^1_\iw(k_\infty,T) \xrightarrow{\; \col^{\pm}} \Lambda_\ccO \simeq \Lambda^{\oplus 2}
\end{equation}
 is a $\Lambda$-homomorphism between two free $\Lambda$-modules of rank $2$ (see Proposition \ref{consequenceofcotorsionness}). We write the composition of this map with projection to the two coordinates as
 \[
 \col^\pm_i: H^1_{\iw}(K_\infty,T) \rightarrow \Lambda,\quad i=1,2.
 \]

\begin{definition}
    Let $\bu=(u_1,u_2)\in H^1_{\iw}(K_\infty,T) ^{\oplus 2}$ and $n\ge0$ be an integer. We define
    \[
    F_n(\bu)=\begin{bmatrix}
     \tomg_n^+\col^-_1(u_1)+\tomg_n^-\col_1^+(u_1) & \tomg_n^+\col^-_1(u_2)+\tomg_n^-\col_1^+(u_2)  \\
  \tomg_n^+\col^-_2(u_1)+\tomg_n^-\col_2^+(u_1)   &  \tomg_n^+\col^-_2(u_2)+\tomg_n^-\col_2^+(u_2)
    \end{bmatrix}\in\mat.
    \]
    Furthermore, we define for $\bullet\in\{+,-\}$
    \[
    \underline{\col}^\bullet(\bu)=\begin{bmatrix}
        \col_1^\bullet(u_1)&\col_1^\bullet(u_2)\\
        \col_2^\bullet(u_1)&\col_2^\bullet(u_2)
    \end{bmatrix}\in\mat.
    \]
\end{definition}

\begin{lemma}\label{lem:congruent-Fu}
    Let $\bu=(u_1,u_2)\in H^1_{\iw}(K_\infty,T) ^{\oplus 2}$ and $n\ge0$ be an integer. For $0\le m\le n$, there exists a non-zero scalar $c_{m,n}\in\Qp(\zeta_{p^m})^\times$ such that
    $F_n(\bu)(\epsilon_m)$ is equal to $c_{m,n}  \underline{\col}^+(\bu)(\epsilon_m) $ or $c_{m,n}  \underline{\col}^-(\bu)(\epsilon_m) $, depending on the parity of $m$.
\end{lemma}
\begin{proof}
Recall that $\tomg_n^+\tomg_n^-=\displaystyle\prod_{1\le m\le n}\Phi_m$. Therefore, for each $1\le m\le n$, exactly one of the two elements of $\{\tomg_n^+(\epsilon_m),\tomg_n^-(\epsilon_m)\}$ vanishes, from which the assertion for $m\ge1$ follows.

If $m=0$, we have $\epsilon_m=0$. Proposition~\ref{surjectivityofcolemanmaps} tells us that $\col^+_i(u_j)(0)=0$. Therefore,
\[
F_n(\bu)(0)=\tomg_n^+(0)  \underline{\col}^-(\bu),
\]
as desired.
\end{proof}

\begin{lemma}\label{lem:nonzero}
Let $\bullet\in\{+,-\}$. Suppose that $\ccX^\bullet(E/K_\infty)$ is $\Lambda$-torsion.
If $\bu=(u_1,u_2)\in H^1_{\iw}(K_\infty,T)^{\oplus 2}$ such that the quotient of $\Lambda$-modules
$H^1_{\iw}(K_\infty,T)\big/\langle u_1,u_2\rangle$ is $\Lambda$-torsion, then $\det  \underline{\col}^\bullet(\bu)\ne 0$.
\end{lemma}
\begin{proof}
     By \cite[Proposition A.3.2]{Per00}, we have the following Poitou--Tate exact sequence
 \begin{equation*}
     H^1_{\iw}(K_\infty,T)  \longrightarrow \frac{H^1_\iw(k_\infty,T)}{\ker \col^\bullet}   \longrightarrow\ccX^\bullet(E/K_\infty) \longrightarrow \ccX^0(E/K_\infty) \longrightarrow 0.
\end{equation*}
Hence, the result follows from the assumption that $\ccX^\bullet(E/K_\infty)$ is $\Lambda$-torsion.
\end{proof}
 
    \begin{prop}\label{prop:good-basis}
         Suppose that Hypothesis~\ref{cotorsionconjecture} holds.
Then, there exists $\bu=(u_1,u_2)\in H^1_{\iw}(K_\infty,T)^{\oplus 2}$ such that 
$H^1_{\iw}(K_\infty,T) \big/\langle u_1,u_2\rangle$ is $\Lambda$-torsion and that $F_n(\bu)$ is a special matrix relative to $n$ (in the sense of Definition~\ref{def:app}) for all $n\ge0$.
    \end{prop}
    \begin{proof}
        Let $\bfz=(z_1,z_2)\in H^1_{\iw}(K_\infty,T)^{\oplus 2}$ such that $\{z_1,z_2\}$ be a $\Lambda$-basis of  $H^1_{\iw}(K_\infty,T)$.  We write
        $$\ccS_n=\left\{m: \rank\, F_n(\bfz)(\epsilon_m)=1,0\le m\le n\right\}.$$
        It follows from Lemma~\ref{lem:congruent-Fu} that 

\[
\ccS_n\subset\left\{m:\Phi_m{\big|}\det\underline\col^+(\bfz)\text{ or }\Phi_m{\big|}\det\underline\col^-(\bfz),0\le m\le n\right\}.
\]
In particular,
         \[
          \ccS_n\subseteq\ccS_{n+1}.
          \]
          Let $\displaystyle\ccS_\infty=\bigcup_{n\ge0}\ccS_n$. 
          As $\det  \underline{\col}^\pm(\bfz)\ne 0$ by Lemma~\ref{lem:nonzero}, the cardinality of $\ccS_n$ is bounded above as $n\rightarrow\infty$ and the cardinality of $\ccS_\infty$ is finite.
        
Let $n\ge0$ be an integer.    Let $m\in\ccS_\infty$ with $0\le m\le n$. Let $\bullet$ be the unique element of $\{+,-\}$ such that
    \[
F_n(\bfz)(\epsilon_m)= c_{m,n}  \underline{\col}^\bullet(\bfz)(\epsilon_m)
    \]
    as given by Lemma~\ref{lem:congruent-Fu}. By linear algebra, there exists a $2\times 2$ invertible matrix $B_m$ over $\Qp(\zeta_{p^m})$ such that $$  \underline{\col}^\bullet(\bfz)(\epsilon_m)B_m=\begin{bmatrix}
        0&*\\
        0&*
    \end{bmatrix},$$
    where the second column is non-zero. Applying Lemma~\ref{lem:CRT}(ii) to each of the four entries of $B_m$ as $m$ runs through $\ccS_\infty$ and multiplying by a power of $p$, if necessary, there exists a $2\times 2$ matrix $B$ defined over $\Zp[X]$ such that $B(\epsilon_m)$ is a non-zero scalar multiple of $B_m$ for all $m\in\ccS_\infty$. Consequently,
    \begin{equation}
    F_n(\bfz)B\equiv\begin{bmatrix}
        0&*\\0&*
    \end{bmatrix}\mod \Phi_m
        \label{eq:rank1}
    \end{equation}
    for all $m\in \ccS_n$. 

    \vspace{2mm}
    
    We claim that we can choose $B$ independently of $n$ satisfying \eqref{eq:rank1} with $\Phi_m\nmid \det(B)$ for all $m\ge0$. Let $M=\max\ccS_\infty$. We already know that $\Phi_m\nmid \det(B)$ for $m\in\ccS_\infty$. We can ensure that $B\mod \Phi_m$ is an invertible matrix for all $m\le M$ if we enlarge the set $\ccS_\infty$ to $\ccS=\{0,\dots, M\}$ and take $B_m$ to be any invertible matrices for $m\in\ccS\setminus\ccS_\infty$ when we apply Lemma~\ref{lem:CRT}(ii). Once such a $B$ is chosen,  any other $B'$ such that $B\equiv B'\mod\omega_M$ satisfies the same congruence relation modulo $\Phi_m$ as $B$ for all $m\in\ccS_\infty$. We may choose recursively a sequence of matrices $\{B_m\}_{m\ge M}$ such that $B_M$ is our initial choice of $B$, and $B_{m+1}\equiv B_{m}\mod\omega_m$ and $\Phi_{m+1}\nmid\det B_{m+1}$. Indeed, suppose that we have chosen 
    \[
    B_m=\begin{bmatrix}
        a&b\\c&d
    \end{bmatrix}
    \] for some $m\ge M$. If $\Phi_{m+1}\nmid \det B_m$, then we can simply choose $B_{m+1}=B_m$. If $\Phi_{m+1} \mid \det B_m$, the vectors
    \[
\begin{bmatrix}
     a(\epsilon_{m+1})\\b(\epsilon_{m+1})
\end{bmatrix},\quad \begin{bmatrix}
     c(\epsilon_{m+1})\\d(\epsilon_{m+1})
\end{bmatrix}\]
are linearly dependent. There exist $\alpha,\beta,\gamma,\delta\in\Zp$ such that
    \[
\begin{bmatrix}
     a(\epsilon_{m+1})+\alpha\omega_m(\epsilon_{m+1})\\b(\epsilon_{m+1})+\beta\omega_m(\epsilon_{m+1}))
\end{bmatrix},\quad \begin{bmatrix}
     c(\epsilon_{m+1})+\gamma\omega_m(\epsilon_{m+1})\\d(\epsilon_{m+1})+\delta\omega_m(\epsilon_{m+1}))
\end{bmatrix}\]
are linearly independent. Therefore, we can choose
\[
B_{m+1}=B_m+\omega_m\begin{bmatrix}
    \alpha&\gamma\\\beta&\delta
\end{bmatrix}.
\]
In particular, our claim follows from taking the limit of this sequence.

For this choice of matrix $B$, we have
\[
    F_n(\bfz)B\equiv\begin{bmatrix}
        0&0\\0&0
    \end{bmatrix}\mod \Phi_m
    \]
    whenever $F_n(\bfz)\mod\Phi_m$ is the zero matrix, {and 
$    F_n(\bfz)B(\epsilon_m)$ is invertible whenever $\Phi_m\nmid \det\left( F_n(\bfz)\right)$.} Together with \eqref{eq:rank1}, we deduce that $ F_n(\bfz)B$ is special relative to $n$ for all $n\ge0$. The assertion of the proposition now follows from taking
\[
u_1=az_1+bz_2,\quad u_2=cz_1+dz_2,
\]
    where $B=\begin{bmatrix}
        a&c\\ b&d
    \end{bmatrix}$.
\end{proof}

 From now on, we fix $\bu=(u_1,u_2)$ satisfying the properties given by Proposition~\ref{prop:good-basis}. For an integer $n\ge1$, we define the following modules
 \[
 M_n=\frac{\Ltwo_n}{\langle F_n(\bu)\rangle_n},\quad M_{n-1}=\frac{\Ltwo_{n-1}}{\langle F_n(\bu)\rangle_{n-1}},
 \]
 where $\langle F_n(\bu)\rangle_m$ denotes the $\Lambda_m$-submodule of $\Lambda_m^{\oplus 2}$ generated by the columns of the matrix $F_n(\bu)$.

\begin{lemma} \label{lem:M_n}
 Suppose that Hypothesis~\ref{cotorsionconjecture} holds. For $n\gg0$, we have
\begin{equation*}
\nabla M_n =\begin{cases}
         2  \deg \tomg_n^+ + \ord_{\epsilon_n}\left(\det  \underline{\col}^-(\bu)(\epsilon_n)\right) & \;\;\; \text{if } n  \text{ is odd,}  \\
          2 \deg\tomg_n^- +\ord_{\epsilon_n}\left(\det  \underline{\col}^+(\bu)(\epsilon_n)\right) & \;\;\; \text{if } n \text{ is even.}
        \end{cases}
\end{equation*}
\end{lemma}
\begin{proof}
Suppose that $n$ is odd, in which case $\tomg_n^-(\epsilon_n)=0$. We have
\begin{align*}
     \det F_n(\bu)(\epsilon_n)&=  \det\begin{bmatrix}
     \tomg_n^+(\epsilon_n)\col^-_1(u_1)(\epsilon_n) & \tomg_n^+(\epsilon_n)\col^-_1(u_2)(\epsilon_n)\\
  \tomg_n^+(\epsilon_n)\col^-_2(u_1)(\epsilon_n)   &  \tomg_n^+(\epsilon_n)\col^-_2(u_2)(\epsilon_n)
    \end{bmatrix}\\
    &=\tomg_n^+(\epsilon_n)^2\det\begin{bmatrix}
     \col^-_1(u_1)(\epsilon_n) & \col^-_1(u_2)(\epsilon_n)\\
 \col^-_2(u_1)(\epsilon_n)   &  \col^-_2(u_2)(\epsilon_n)
    \end{bmatrix}\\
    &=\tomg_n^+(\epsilon_n)^2 \det\underline{\col}^-(\bu)(\epsilon_n).
\end{align*}
 Lemma~\ref{lem:nonzero} tells us that $\det\underline{\col}^-(\bu)\ne0$. In particular, when $n$ is sufficiently large, we have $\Phi_n\nmid\det F_n(\bu)$. Therefore, combining Theorem~\ref{thm:app} and Proposition~\ref{prop:good-basis} gives
 \[
 \nabla M_n=2\,\ord_{\epsilon_n}\tomg_n^+(\epsilon_n)+\ord_{\epsilon_n}\left(\det\underline{\col}^-(\bu)(\epsilon_n)\right).
 \]
 Since $\tomg_{n}^+$ is a distinguished polynomial of degree $<p^{n}-p^{n-1}$ for $n>1$, it follows that $\ord_{\epsilon_n}\tomg_n^+(\epsilon_n)=\deg\tomg_n^+$, giving the desired formula. The case where $n$ is even can be proved in the same way.
\end{proof}

\subsection{Final steps of the proof}

We study the variation of  $\ccY(E/K_n)$ defined at the beginning of the section using the following $\Zp$-modules:
\begin{align*}
    \ccY'(E/K_n)&:=\coker \left( H^1_\iw(K_\infty,T)_{\Gamma_n} \longrightarrow \frac{H^1(k_n,T)}{E(k_n)\otimes\bbZ_p} \right),\\
\ccY''(E/K_n)&:=\coker \left( H^1_\iw(K_\infty,T)_{\Gamma_n} \longrightarrow \frac{H^1(k_n,T)}{H^1_+(k_n,T)}\oplus \frac{H^1(k_n,T)}{H^1_-(k_n,T)} \right).
\end{align*}

        \vspace{2mm}

\begin{lemma}\label{kobayashirankofy'andy''}
    For all $n\ge 0$, we have $\ccY'(E/K_n)=\ccY''(E/K_n)$.
\end{lemma}
\begin{proof}
It follows from Theorem \ref{generatingtheorem}(i) that there is an isomorphism of $\Lambda$-modules
\[
\hatE(\fram_n)\otimes\Qp/\Zp\simeq \left(\hatE(\fram_n)^+\otimes\Qp/\Zp\right)\oplus\left(\hatE(\fram_n)^-\otimes\Qp/\Zp\right).
\]
On taking Pontryagin duals, we deduce:
\[  \frac{H^1(k_n,T)}{E(k_n)\otimes\bbZ_p}\simeq  \frac{H^1(k_n,T)}{H^1_+(k_n,T)}\oplus \frac{H^1(k_n,T)}{H^1_-(k_n,T)},
\]
which implies the assertion.
\end{proof}

\begin{lemma}\label{lem:Rod}
 Let $B\in \mat$ such that $\det(B)$ is coprime to $\omega_n$ for some integer $n\ge0$. Then $$\omega_n\Lambda^{\oplus 2}\cap \langle B\rangle=\omega_n\langle B\rangle,\quad\text{and}\quad \left(\frac{\Lambda^{\oplus 2}}{\langle B\rangle}\right)^{\Gamma_n}=0 .$$   
\end{lemma}

\begin{proof}
It is clear that $\omega_n\Lambda^{\oplus 2}\cap \langle B\rangle\supseteq\omega_n\langle B\rangle$.
    Let $x$ be an element of the intersection. There exist $y\in \Lambda^{\oplus 2}$ and $a,b\in\Lambda$ such that
    \[
    x=\omega_n y=a B_1+bB_2,
    \]
    where $B_1$ and $B_2$ are the columns of $B$. In particular,
    \[
    a(\epsilon_m)B_1(\epsilon_m)+b(\epsilon_m)B_2(\epsilon_m)=0
    \]
    for $0\le m\le n$. But $B_1(\epsilon_m)$ and $B_2(\epsilon_m)$ are linearly independent vectors in $\Qp(\zeta_{p^m})^{\oplus 2}$, which implies that $a(\epsilon_m)=b(\epsilon_m)=0$. In other words, both $a$ and $b$ are divisible by $\Phi_m$. Consequently, both elements are divisible by $\omega_n$. Thus, $x\in \omega_n\langle B\rangle$, proving the first equality.

Suppose that $x+\langle B\rangle\in \left(\frac{\Lambda^{\oplus 2}}{\langle B\rangle}\right)^{\Gamma_n}$ for some $x\in\Lambda^{\oplus 2}$. Then $\gamma^{p^n} x\equiv x\mod \langle B\rangle$ since $\gamma^{p^n}$ is a topological generator of $\Gamma_n$. In particular, $\omega_n x\in \langle B\rangle$, which implies  $$\omega_n x\in\omega_n\Lambda^{\oplus 2}\cap \langle B\rangle=\omega_n\langle B\rangle.$$
Therefore, $x\in\langle B\rangle$, which gives the second desired equality.

\end{proof}
\vspace{2mm}

\begin{prop}\label{prop:nablaY}
 Suppose that Hypothesis~\ref{cotorsionconjecture} holds. Let $\{z_1,z_2\}$ be a $\Lambda$-basis of $H^1_{\iw}(K_\infty,T)$ and write $\bfz=(z_1,z_2)$. For $n\gg0$, we have
\begin{equation*}
\nabla \ccY(E/K_n) =\begin{cases}
         2  \deg\tomg_n^+ + \ord_{\epsilon_n}\left(\det  \underline{\col}^-(\bfz)(\epsilon_n)\right) & \;\;\; \text{if } n  \text{ is odd,}  \\
          2 \deg\tomg_n^- +\ord_{\epsilon_n}\left(\det  \underline{\col}^+(\bfz)(\epsilon_n)\right) & \;\;\; \text{if } n \text{ is even.}
        \end{cases}
\end{equation*}
\end{prop}
\begin{proof}
    Let $\Lambda_n^-=\Lambda\big/(\omega_n^-)$ and $I_n^+=X\Lambda_n^+=X\Lambda\big/(\omega_n^+)$.
Consider the following commutative diagram with exact rows
\begin{center}
  \begin{tikzpicture}
        \node (Q1) at (-3.5,2) {$0$};
    \node (Q2) at (-1,2) {$\left( I_{n}^{+}\oplus\Lambda_{n}^{-}  \right)^{\oplus 2}$};
    \node (Q3) at (3,2) {$\Lambda_{n}^{\oplus 2}$};    
     \node (Q4) at (7,2) {$ \left(\Lambda_{n}\big/ \left(\tomg_n^+,\omega_n^- \right)\right)^{\oplus 2}$};
      \node (Q5) at (10,2) {$0$};

       \node (Q6) at (-3.5,0) {$0$};      
      \node (Q7) at (-1,0) {$\left(  I_{n-1}^{+}\oplus\Lambda_{n-1}^{-}\right)^{\oplus 2}$};
    \node (Q8) at (3,0) {$\Lambda_{n-1}^{\oplus 2}$};
\node (Q9) at (7,0) {$ \left(\Lambda_{n-1}\big/ \left(\tomg_n^+,\omega_n^- \right)\right)^{\oplus 2}$};
   \node (Q10) at (10,0) {$0$,};

 \draw[->] (Q1)--(Q2);
  \draw[->] (Q2)--node [above] {$\ccF_n^{\oplus 2}$}(Q3);
  \draw[->] (Q2)--node [right] {$\pr$}(Q7);
  \draw[->] (Q3)--node [right] {$\pr$}(Q8);
  \draw[->] (Q4)--node [right] {$\pr$}(Q9);
  \draw[->] (Q3)--(Q4);
  \draw[->] (Q4)--(Q5);
   \draw[->] (Q6)--(Q7);
    \draw[->] (Q7)--node [above] {$\hat \ccF_{n}^{\oplus 2}$}(Q8);
   \draw[->] (Q8)--(Q9);
     \draw[->] (Q9)--(Q10);
         
       \end{tikzpicture}
\end{center}
where  $\ccF_n: I_{n}^{+} \oplus  \Lambda_{n}^{-} \longrightarrow \Lambda_n$ is given by $ (f,g)\mapsto \tomg_n^-f +\tomg_n^+g\pmod{\omega_n}$, and $  \hat \ccF_n:   I_{n-1}^{+}\oplus\Lambda_{n-1}^{-}  \longrightarrow \Lambda_{n-1}$ is the map $ (f,g)\mapsto \tomg_n^-f+\tomg_n^+g\pmod{\omega_{n-1}}$. This is well-defined since $\tomg_n^\pm\omega_n^\mp=\omega_n$.  

It follows from \cite[Lemma 10.7 i)]{kob03} that the third vertical map is an isomorphism. Furthermore, Proposition \ref{surjectivityofcolemanmaps} tells us that the Coleman maps $\col^\pm_n$ give the following identifications
\begin{equation*}
  \frac{H^1(k_n,T)}{H^1_+(k_n,T)}  \simeq I_{\ccO,n}^+\simeq I_{n}^{+,\oplus 2} \quad\text{ and }    \quad\frac{H^1(k_n,T)}{H^1_-(k_n,T)} \simeq \Lambda
_{\ccO,n}^-\simeq\Lambda_{n}^{-,\oplus 2}.
\end{equation*}
Therefore, we may rewrite the commutative diagram above as
\begin{center}
  \begin{tikzpicture}
        \node (Q1) at (-4.5,2) {$0$};
    \node (Q2) at (-1,2) {$\frac{H^1(k_n,T)}{H^1_+(k_n,T)}\oplus \frac{H^1(k_n,T)}{H^1_-(k_n,T)}$};
    \node (Q3) at (3,2) {$\Lambda_{n}^{\oplus 2}$};    
     \node (Q4) at (7,2) {$ \left(\Lambda_{n}\big/ \left(\tomg_n^+,\omega_n^- \right)\right)^{\oplus 2}$};
      \node (Q5) at (10,2) {$0$};

       \node (Q6) at (-4.5,0) {$0$};      
      \node (Q7) at (-1,0) {$\frac{H^1(k_{n-1},T)}{H^1_+(k_{n-1},T)}\oplus \frac{H^1(k_{n-1},T)}{H^1_-(k_{n-1},T)}$};
    \node (Q8) at (3,0) {$\Lambda_{n-1}^{\oplus 2}$};
\node (Q9) at (7,0) {$ \left(\Lambda_{n-1}\big/ \left(\tomg_n^+,\omega_n^- \right)\right)^{\oplus 2}$};
   \node (Q10) at (10,0) {$0$,};

 \draw[->] (Q1)--(Q2);
  \draw[->] (Q2)--node [above] {}(Q3);
  \draw[->] (Q2)--node [right] {}(Q7);
  \draw[->] (Q3)--node [right] {}(Q8);
  \draw[->] (Q4)--node [right] {$\simeq$}(Q9);
  \draw[->] (Q3)--(Q4);
  \draw[->] (Q4)--(Q5);
   \draw[->] (Q6)--(Q7);
    \draw[->] (Q7)--node [above] {}(Q8);
   \draw[->] (Q8)--(Q9);
     \draw[->] (Q9)--(Q10);
         
       \end{tikzpicture}
\end{center}
where the first two horizontal maps are given by the compositions of $\ccF_n^{\oplus2}$ with $(\col^+_n,\col^-_n)$, and $\hat\ccF_n^{\oplus2}$ with $(\col^+_{n-1},\col^-_{n-1})$, respectively.
Let $\pr_n:H^1_\iw(K_\infty,T)\rightarrow H^1_\iw(K_\infty,T)_{\Gamma_n}$ denote the natural projection. If we write
  \[
 \ccY'''(E/K_n)=\coker\left(\pr_n\left(\langle u_1,u_2\rangle\right)\longrightarrow \frac{H^1(k_n,T)}{H^1_+(k_n,T)}\oplus \frac{H^1(k_n,T)}{H^1_-(k_n,T)} \right),
 \]
we have the following commutative diagram with exact rows:

  \begin{center}
  \begin{tikzpicture}
   \node (Qa) at (-3.5,2) {$0$};
    \node (Qb) at (-3.5,0) {$0$};
    \node (Qc) at (10,0) {$0$.};
    \node (Qd) at (10,2) {$0$};

     \node (Q2) at (-1,2) {$\ccY'''(E/K_n)$};
     \node (Q3) at (6,2) { $\left(\Lambda_{n}\big/ \left(\tomg_n^+,\omega_n^- \right)\right)^{\oplus 2}$};
      
      \node (Q5) at (-1,0) {$\ccY'''(E/K_{n-1})$};
    \node (Q6) at (6,0) {$ \left(\Lambda_{n-1}\big/ \left(\tomg_n^+,\omega_n^- \right)\right)^{\oplus 2}$};
\node (Q7) at (2.5,0) {$M_{n-1}$};
    \node (Q8) at (2.5,2) {$M_n$};

    \draw[->] (Q2)--(Q8);
    \draw[->] (Q5)--(Q7);
      \draw[->] (Q7)--(Q6);
    \draw[->] (Q8)--(Q3);
     \draw[->] (Q8)-- (Q7);
    \draw[->] (Q3)--node [right] {$\simeq$} (Q6);
     \draw[->] (Qa)--(Q2);
     \draw[->] (Q3)--(Qd);
      \draw[->] (Qb)--(Q5);
     \draw[->] (Q6)--(Qc);
      \draw[->] (Q2)--(Q5);
       \end{tikzpicture}
\end{center}
Therefore, $\nabla\ccY'''(E/K_n)=\nabla M_n$.

It follows from Lemma~\ref{lem:nonzero} that the map
\begin{align*}
  H^1_\iw(K_\infty,T)_{\Gamma_n} \longrightarrow \frac{H^1(k_n,T)}{H^1_+(k_n,T)}\oplus \frac{H^1(k_n,T)}{H^1_-(k_n,T)}
\end{align*}
is injective for $n\gg0$. For such $n$, there is a short exact sequence
\begin{equation}
    0\longrightarrow \frac{H^1_\iw(K_\infty,T)_{\Gamma_n}}{\pr_n\langle u_1,u_2\rangle}\longrightarrow \ccY'''(E/K_n)\longrightarrow\ccY''(E/K_n)\longrightarrow 0.
\label{eq:relationgYs}
\end{equation}

Let $B\in\mat$ be the $2\times2$ matrix given by the equation $$\begin{bmatrix}
    u_1&u_2
\end{bmatrix}=\begin{bmatrix}
    z_1&z_2
\end{bmatrix}B$$
as in the proof of Proposition~\ref{prop:good-basis}. Recall that $H^1_\iw(K_\infty,T)\simeq\Lambda^{\oplus 2}$ (see Proposition~\ref{consequenceofcotorsionness}(i)) and $B$ is chosen so that $\det(B)$ is coprime to $\omega_n$ for all $n$. Lemma~\ref{lem:Rod} implies that $$\pr_n\langle u_1,u_2\rangle\simeq\frac{\langle B\rangle+\omega_n\Lambda^{\oplus 2}}{\omega_n\Lambda^{\oplus 2}}\simeq\frac{\langle B\rangle}{\langle B\rangle\cap\omega_n\Lambda^{\oplus 2}}=\frac{\langle B\rangle}{\omega_n\langle B\rangle}\simeq\langle u_1,u_2\rangle_{\Gamma_n},$$
and
$$\left(\frac{H^1_\iw(K_\infty,T)}{\langle u_1,u_2\rangle}\right)^{\Gamma_n}\simeq \left(\frac{\Lambda^{\oplus 2}}{\langle B\rangle}\right)^{\Gamma_n}=0.$$
Combined with \cite[Lemma~10.3]{kob03} gives the short exact sequence
\[
0\longrightarrow\langle u_1,u_2\rangle_{\Gamma_n}\longrightarrow H^1_\iw(K_\infty,T)_{\Gamma_n}\longrightarrow \left(\frac{H^1_\iw(K_\infty,T)}{\langle u_1,u_2\rangle}\right)_{\Gamma_n}\longrightarrow0,
\]
and we deduce that the first term of \eqref{eq:relationgYs} is isomorphic to $\displaystyle\left(\frac{H^1_{\iw}(K_\infty,T) }{\langle u_1,u_2\rangle}\right)_{\Gamma_n}$. Since $\det(B)$ generates the $\Lambda$-characteristic ideal of $\displaystyle\frac{H^1_{\iw}(K_\infty,T) }{\langle u_1,u_2\rangle}$, we have for $n\gg0$,
\[
\nabla\frac{H^1_{\iw}(K_\infty,T)_{\Gamma_n}}{\pr_n\left(\langle u_1,u_2\rangle\right)}=\ord_{\epsilon_n}\left(\det(B)(\epsilon_n)\right)
\]
by Lemma~\ref{kobayashirankandcharacteristicpolynomial}(ii). Thus, for $n\gg0$, $\nabla\ccY''(E/K_n)$ is defined and is equal to $\nabla M_n-\ord_{\epsilon_n}\left(\det(B)(\epsilon_n)\right)$.
Furthermore, $\det  \underline{\col}^\pm(\bu)=\det  \underline{\col}^\pm(\bfz)\det(B)$ by definition. Hence, after combining Lemmas~\ref{kobayashirankofy'andy''} and~\ref{lem:M_n}, we deduce that
\begin{equation*}
\nabla \ccY'(E/K_n) =\begin{cases}
         2  \deg\tomg_n^+ + \ord_{\epsilon_n}\left(\det  \underline{\col}^-(\bfz)(\epsilon_n)\right) & \;\;\; \text{if } n  \text{ is odd,}  \\
          2 \deg\tomg_n^- +\ord_{\epsilon_n}\left(\det  \underline{\col}^+(\bfz)(\epsilon_n)\right) & \;\;\; \text{if } n \text{ is even.}
        \end{cases}
\end{equation*}
It follows from an adaptation of \cite[Proposition~10.6 i)]{kob03} and \cite[Proposition~3.6]{LL22} to the current setting that for $n\gg 0$, 
    \begin{equation*}
        \nabla \ccY(E/K_n)=\nabla \ccY'(E/K_n),
    \end{equation*}
    which concludes the proof.
\end{proof}

\vspace{2mm}

 \begin{corollary}\label{cor:nablaX}
 Suppose that Hypothesis~\ref{cotorsionconjecture} holds.     For $n\gg 0$, we have
    \begin{equation*}
       \nabla \ccX(E/K_n)=
\begin{cases}
  2 s_{n-1}   + \lambda_- + (p^n-p^{n-1})\mu_- & \;\;\; \text{if } n  \text{ is odd,}  \\
        2s_{n-1}+\lambda_+ + (p^n-p^{n-1})\mu_+ & \;\;\; \text{if } n \text{ is even,}
        \end{cases}
    \end{equation*}
    where $s_n=\sum_{k=1}^{n}(-1)^{n-k}p^k$ for $n\geq 0$.
  \end{corollary}
 
 \begin{proof}
Let $\bullet\in\{+,-\}$. Let $f_\bullet\in\Lambda$ (resp. $f_0$) denote a characteristic element of the torsion $\Lambda$-module $\ccX^\bullet(E/K_\infty)$ (resp. $\ccX^0(E/K_\infty)$).  We have the following Poitou--Tate exact sequence of $\Lambda$-torsion modules:
 \[
 0\longrightarrow \frac{\image\left(\underline\col^\bullet\right)}{\langle  \underline{\col}^\bullet(\bfz)\rangle}\longrightarrow \ccX^\bullet(E/K_\infty)\longrightarrow\ccX^0(E/K_\infty)\longrightarrow0.
 \]
 As $\image(\col^-)=\Ltwo$ and $\image(\col^+)=(X\Lambda)^{\oplus2}$, this gives the equality of $\Lambda$-ideals
 \[
 (X^{n^\bullet}\det  \underline{\col}^\bullet(\bfz)f_0)=(f^\bullet),
 \]
 where $n^-=0$ and ${n^+=-2}$.
 Lemma~\ref{kobayashirankofX0} tells us that
 \[
 \nabla \ccX(E/K_n)=\nabla\ccY(E/K_n)+\nabla\ccX^0(E/K_n).
 \]
 Hence, the corollary follows from Lemma~\ref{kobayashirankandcharacteristicpolynomial} and Proposition~\ref{prop:nablaY}.
 \end{proof}

\vspace{2mm}

We are now ready to conclude the proof of Theorem~\ref{maintheorem}.

\vspace{2mm}

\begin{theorem}
  Suppose that Hypothesis~\ref{cotorsionconjecture} holds. Under our running hypotheses, we have:
    \begin{enumerate}[label=(\roman*)]
        \item $\rank_\bbZ E(K_n)$ is bounded independently of $n$;
        
\vspace{1.5mm}
        
        \item Assume that $\Sha(E/K_n )[p^\infty]$ is finite for all $n\geq 0$. Define $\displaystyle r_\infty=\sup_{n\ge0}\left\{ \rank_\bbZ E(K_n)\right\}$. Then, for $n\gg 0$, we have
        \[\nabla\Sha(E/K_n )[p^\infty]=\begin{cases}
      2s_{n-1}   + \lambda_- + \phi(p^n)\mu_-   -r_\infty   & \;\;\; \text{if } n  \text{ is odd,}  \\
          2s_{n-1} +\lambda_+ + \phi(p^n)\mu_+ -r_\infty   & \;\;\; \text{if } n \text{ is even,}
        \end{cases} \]
        where $\phi(p^n)=p^n-p^{n-1}$, and $s_n=\sum_{k=1}^{n}(-1)^{n-k}p^k$, for $n\geq 0$.
    \end{enumerate}
\end{theorem}
\begin{proof}
      Since $\nabla\ccX(E/K_n)$ is defined when $n$ is sufficiently large by Corollary~\ref{cor:nablaX}, the kernel and cokernel of the natural map $\ccX(E/K_{n+1})\rightarrow \ccX(E/K_n)$ are finite for such $n$. In particular, $\rank_\bbZ E(K_n)$ is bounded independently of $n$. Hence, the first assertion follows from the well-known exact sequence
    \begin{equation}\label{selmergroupexactsequence}
            0 \longrightarrow E(K_n)\otimes \bbQ_p/\bbZ_p \longrightarrow \Sel_{p^\infty}(E/K_n) \longrightarrow \Sha(E/K_n )[p^\infty] \longrightarrow 0.
        \end{equation}
        
    It follows from the exact sequence \eqref{selmergroupexactsequence} and Lemma \ref{kobayashirankzero} that
    \begin{equation*}
        \nabla\Sha(E/K_n )[p^\infty] = \nabla \ccX(E/K_n)-\nabla E(K_n)\otimes \bbZ_p.
    \end{equation*}
    Part (i) tells us that $\nabla E(K_n)\otimes \bbZ_p=r_\infty$ for $n\gg0$.  Thus, the assertion (ii) follows from Corollary~\ref{cor:nablaX}.
\end{proof}
\vspace{2mm}
\begin{remark}
   Similarly to \cite[Corollary~10.2]{kob03}, it would be possible to give an alternative proof of part (i) of Theorem~\ref{maintheorem} via analogues of the control theorems in \S9 of \textit{op. cit.}
\end{remark}

\bibliographystyle{alpha} 
\bibliography{reference}

\end{document}